\documentclass[11pt]{amsart}

\usepackage[T1]{fontenc}
\usepackage[utf8]{inputenc}
\usepackage{lmodern}
\usepackage{microtype}
\usepackage{amsmath,amssymb,amsthm,mathtools}
\usepackage{booktabs}
\usepackage{array}
\usepackage{longtable}
\usepackage{enumitem}
\usepackage{xcolor}
\usepackage[hidelinks]{hyperref}
\usepackage[nameinlink,capitalize]{cleveref}

\newtheorem{theorem}{Theorem}[section]
\newtheorem{proposition}[theorem]{Proposition}
\newtheorem{lemma}[theorem]{Lemma}

\theoremstyle{definition}
\newtheorem{definition}[theorem]{Definition}

\newcommand{\F}{\mathbb F}

\newcommand{\HH}{H_{343}}
\newcommand{\dd}{\mathsf d}

\newcommand{\Om}{\Omega}
\newcommand{\Sig}{\Sigma}

\title[The small Davenport constant of $H_{343}$]{The small Davenport constant of the Heisenberg group of order $343$}

\author{Andreas Volkmann}
\address{Independent researcher, Seligenstadt, Germany}
\email{volkmannandreas7@gmail.com}
\date{August 2026}


\begin{document}

\begin{abstract}
For a finite group $G$, let $\dd(G)$ denote the maximum length of a
sequence having no nonempty subsequence whose terms can be ordered to
have product one.  For an odd prime $p$, let
$H_{p^3}=\operatorname{UT}_3(\F_p)$.  Godara and Sarkar proved
$\dd(H_{27})=6$ and conjectured $\dd(H_{p^3})=3p-3$; in a recent
preprint, White proved the next case $\dd(H_{125})=12$ and left
$18\leq\dd(H_{343})\leq24$.  We prove $\dd(H_{343})=18$.

We adopt White's product-one criterion and spread framework and develop
a $p=7$-specific direction stratification.  An explicit
product-one-free sequence gives the lower bound.  For the upper bound,
we stratify a hypothetical product-one-free sequence of length $19$ by
the number of central terms and by the occupied projective directions
of its quotient multiset.  Supports on at most two directions are
excluded by a theoretical argument whose finite auxiliary statements
are exhaustively checked; the three-direction case and the case of five
central terms are settled by exact finite computations.  The remaining
thirty strata are encoded by a counterexample-guided SAT procedure.  A
separately implemented checker verifies all $9{,}920{,}815$ seed cuts
and all $27{,}207$ learned cuts, and each final unsatisfiable instance
is accompanied by a checked LRAT certificate.  A separate
implementation-level audit verifies the master encoding, the proof
archives, and the lower-bound witness.
\end{abstract}

\maketitle
\enlargethispage{\baselineskip}

\section{Introduction}

For a finite group $G$, a \emph{sequence over $G$} is a finite unordered list of elements of $G$, with repetition allowed. A nonempty subsequence is \emph{product-one} if its terms can be ordered so that their product is the identity. The small Davenport constant $\dd(G)$ is the maximum length of a product-one-free sequence over $G$.

We consider the Heisenberg group
\[
H_{p^3}=\operatorname{UT}_3(\F_p)
 =\left\{M(a,b,c)=
\begin{pmatrix}
1&a&c\\0&1&b\\0&0&1
\end{pmatrix}:a,b,c\in\F_p\right\},
\]
with multiplication
\[
M(a,b,c)M(a',b',c')=M(a+a',b+b',c+c'+ab').
\]
Write
\[
x=M(1,0,0),\qquad y=M(0,1,0),\qquad v=M(0,0,1).
\]
Then $Z(H_{p^3})=\langle v\rangle$, and $H_{p^3}/Z(H_{p^3})\cong\F_p^2$.
Moreover, with $[g,h]=ghg^{-1}h^{-1}$,
\[
[M(a,b,c),M(a',b',c')]=M(0,0,ab'-a'b).
\]
Thus two noncentral elements commute exactly when their quotient
vectors are collinear.  The projective
directions in $\F_p^2$ correspond to the $p+1$ maximal abelian
subgroups of $H_{p^3}$ that contain the center.

Godara and Sarkar proved $\dd(H_{27})=6$ and posed the conjecture
\[
\dd(H_{p^3})=3p-3
\qquad\text{for every odd prime }p
\]
\cite{GodaraSarkar2025}. In a recent preprint, White proved
$\dd(H_{125})=12$ and identified the precise obstruction that prevents
his $p=5$ argument from extending directly to $p=7$; that work left the
range
\[
18\leq \dd(H_{343})\leq 24
\]
\cite{White2026}.

Our main result settles this next case.

\begin{theorem}\label{thm:main}
For the exponent-$7$ Heisenberg group $\HH=\operatorname{UT}_3(\F_7)$,
\[
\dd(\HH)=18.
\]
\end{theorem}

The proof is computer-assisted but has a narrow and explicit trust boundary. The mathematical reduction is given in \cref{sec:criterion,sec:two-directions}. After exact auxiliary enumerations for the two- and three-direction cases and for $z=5$, the remaining cases with at least four quotient directions and $z\leq4$ form thirty SAT instances. Every monotone exclusion used by those instances is checked by code separate from the generating oracle, and every final UNSAT conclusion has a machine-checkable LRAT certificate. The full certificate package, source code, manifests, and a separate implementation-level audit are deposited in the accompanying Zenodo repository:
\[
\href{https://doi.org/10.5281/zenodo.21833183}{\texttt{10.5281/zenodo.21833183}}.
\]

The product-one criterion, spread invariant, central-forcing argument,
and line cap below are taken from White's framework \cite{White2026}
and are recalled with proofs.  All statements from that framework
needed here are restated and proved in the present paper.  The new
contribution is the closure of the $p=7$ spread problem: a theoretical
treatment of supports on at most two projective directions, exact
certified treatments of the three-direction and high-central strata,
and a proof-producing counterexample-guided SAT resolution of the
remaining thirty strata.

\section{The product-one criterion and a uniform spread reduction}\label{sec:criterion}

Let $p$ be an odd prime and write $H=H_{p^3}$.  For
$g=M(a,b,c)$ put
\[
\pi(g)=(a,b)\in\F_p^2,\qquad c(g)=c\in\F_p.
\]
For an ordered list $I=(g_1,\ldots,g_k)$ with
$\pi(g_i)=(a_i,b_i)$ define
\[
q(I)=\sum_{1\leq i<j\leq k} a_i b_j\in\F_p.
\]
For an unordered block $B$, let
\[
\Om(B)=\{q(I): I\text{ is an ordering of }B\}\subseteq\F_p.
\]

\begin{lemma}[Product-one criterion]\label{lem:criterion}
For every ordering $I=(g_1,\ldots,g_k)$,
\[
\prod_{i=1}^k g_i
=M\!\left(\sum_i a_i,\sum_i b_i,\sum_i c_i+q(I)\right).
\]
Consequently, a block $B$ has a product-one ordering if and only if
\[
\sum_{g\in B}\pi(g)=0
\quad\text{and}\quad
-\sum_{g\in B} c(g)\in\Om(B).
\]
\end{lemma}

\begin{proof}
Induct on the length. Appending $M(a,b,c)$ to a prefix with quotient
sum $(A,B)$ adds $c+Ab$ to the central coordinate, producing exactly
the cross terms $a_i b_j$ with $i<j$.
\end{proof}

Let $S$ be a sequence over $H$.  Write $W=S\cap Z(H)$ for its central
part, let $z=|W|$, and let $Q=\pi(S\setminus Z(H))$ be the multiset of
nonzero quotient vectors.  A \emph{zero-sum block} is a nonempty
submultiset $B\subseteq Q$ with vector sum $0$.
Because $q(I)$ depends only on the quotient vectors, we use the same
notation $\Om(B)$ for a quotient block $B\subseteq Q$.  When such a
block is used in $S$, let $\widetilde B$ denote the corresponding
chosen noncentral group terms; it has the same cross-sum set.

\begin{definition}[Spread]\label{def:spread}
For a quotient multiset $Q$, define
\[
\sigma_p(Q)=\max\sum_{j=1}^m\bigl(|\Om(B_j)|-1\bigr),
\]
where the maximum is over all pairwise disjoint zero-sum blocks
$B_1,\ldots,B_m\subseteq Q$.  The empty family is allowed and has value
$0$.
\end{definition}

The central terms contribute their subset sums.  Let
\[
\Sig_0(W)=\left\{\sum_{w\in U}c(w):U\subseteq W\right\}\subseteq\F_p.
\]

\begin{lemma}[Central forcing]\label{lem:forcing}
If $W$ is zero-sum-free in $C_p$, then
\[
|\Sig_0(W)|\geq z+1.
\]
If, in addition, $B_1,\ldots,B_m$ are pairwise disjoint zero-sum
blocks with $m\geq1$ and
\[
|\Sig_0(W)|+\sum_{j=1}^m\bigl(|\Om(B_j)|-1\bigr)\geq p,
\]
then $S$ has a nonempty product-one subsequence.  In particular, when
$z\leq p-2$, the inequality
\[
\sigma_p(Q)\geq p-1-z
\]
forces a product-one subsequence.
\end{lemma}

\begin{proof}
Repeated use of the Cauchy--Davenport theorem
\cite{Nathanson1996} gives $|\Sig_0(W)|\geq z+1$.
Choose the corresponding pairwise disjoint group blocks
$\widetilde B_1,\ldots,\widetilde B_m$ and put
\[
C=\sum_{j=1}^m\sum_{g\in\widetilde B_j}c(g).
\]
Concatenate independently chosen orderings.  Cross terms between
different blocks vanish because every block has quotient sum $0$.
Cauchy--Davenport again gives
\[
\left|\Om(B_1)+\cdots+\Om(B_m)\right|
\geq \min\left(p,1+\sum_j(|\Om(B_j)|-1)\right).
\]
The assumed inequality therefore yields
\[
C+\Om(B_1)+\cdots+\Om(B_m)+\Sig_0(W)=\F_p.
\]
In particular, an ordering of the group blocks and a subset of $W$
can be chosen with total central coordinate $0$.  Their quotient sum
is also $0$, so \cref{lem:criterion} supplies the claimed nonempty
subsequence.  For the last assertion, $p-1-z\geq1$, so any family
attaining the stated spread is nonempty.
\end{proof}

A \emph{projective direction} is a one-dimensional subspace of
$\F_p^2$.  For a direction $L$, let $n_L$ be the number of noncentral
terms whose quotient lies on $L$.

\begin{lemma}[Line cap]\label{lem:linecap}
If $S$ is product-one-free, then every projective direction satisfies
\[
n_L+z\leq 2p-2.
\]
\end{lemma}

\begin{proof}
The inverse image $\pi^{-1}(L)$ is an abelian subgroup of order $p^2$
and exponent $p$, hence isomorphic to $C_p^2$.  The subsequence
consisting of the $n_L$ terms on $L$ together with the $z$ central
terms lies in this subgroup.  Olson's theorem gives
$\dd(C_p^2)=2p-2$ \cite{Olson1969I,Olson1969II}.
\end{proof}

\begin{proposition}[Uniform spread reduction]\label{prop:uniform}
Let $p$ be an odd prime.  A hypothetical product-one-free sequence
$S$ over $H_{p^3}$ of length $3p-2$ must satisfy
\[
0\leq z\leq p-2,\qquad |Q|=3p-2-z,
\qquad n_L\leq2p-2-z\quad\text{for every }L.
\]
Consequently, the upper bound $\dd(H_{p^3})\leq3p-3$ follows once one
proves
\[
\sigma_p(Q)\geq p-1-z
\]
for every quotient multiset satisfying these conditions.
\end{proposition}

\begin{proof}
The central part of a product-one-free sequence is zero-sum-free in
$C_p$, so $z\leq p-1$.  If $z=p-1$, Cauchy--Davenport gives
$\Sig_0(W)=\F_p$, while the remaining $2p-1$ quotient terms contain a
nonempty zero-sum block by Olson's theorem for $C_p^2$.  Choose any
ordering of that block and then a central subset cancelling its central
coordinate; \cref{lem:criterion} gives a product-one subsequence, a
contradiction.  Hence $z\leq p-2$.  The formula for $|Q|$ is immediate,
and the bounds on $n_L$ are \cref{lem:linecap}.  The final assertion is
\cref{lem:forcing}.
\end{proof}

The reduction above is uniform, but the finite proof of its spread
inequality is specialized below to $p=7$.  In particular, the present
computation is neither an asymptotic algorithm nor a proof of the
conjecture for all odd primes.  From now on, write $H=\HH$ and
$\sigma(Q)=\sigma_7(Q)$.

\section{The lower bound}

\begin{proposition}\label{prop:lower}
For every odd prime $p$, the sequence
\[
x^{p-1}y^{p-1}v^{p-1}
\]
over $H_{p^3}$ is product-one-free.  Hence
$\dd(H_{p^3})\geq3p-3$, and in particular $\dd(\HH)\geq18$.
\end{proposition}

\begin{proof}
A product-one subsequence would use $\alpha$ copies of $x$ and $\beta$
copies of $y$ with $0\leq\alpha,\beta\leq p-1$.  Its quotient sum can
vanish only if $\alpha\equiv\beta\equiv0\pmod p$, hence
$\alpha=\beta=0$.  What remains is $v^\gamma$ with
$1\leq\gamma\leq p-1$, which is not the identity.
\end{proof}

\section{Supports on at most two directions}\label{sec:two-directions}

We now assume that $S$ is product-one-free and $|S|=19$. The case of at most one quotient direction follows immediately from \cref{lem:linecap}. For two directions, write
\[
n_1+n_2+z=19.
\]
The two line caps imply
\[
n_1,n_2\geq7,\qquad z\leq5.
\]
All unordered pairs of line sizes below are understood up to
interchanging the two lines.
After choosing a nonzero coordinate on each line, the quotient values on each line form a multiset in $\F_7^*$.

The following finite facts were enumerated twice by separately written algorithms. A minimal zero-sum multiset is an inclusion-minimal nonempty zero-sum multiset over $C_7$.
On a fixed line, a \emph{pair} and a \emph{triple} mean minimal
zero-sum blocks of sizes $2$ and $3$, respectively; a \emph{double
pair} is a disjoint union of two pairs.  A hyphen below denotes the
disjoint union of blocks on the two independent lines.  For an
opposite-pair support $\{1,-1\}$, write
$P_i=\{1^i,(-1)^i\}$.

\begin{proposition}[Finite two-line facts]\label{prop:Ffacts}
The following statements hold.
\begin{enumerate}[label=(F\arabic*)]
\item There are $47$ minimal zero-sum multisets over $C_7$. Their counts in lengths $2$ through $7$ are, respectively, $3$, $8$, $12$, $12$, $6$, and $6$. Among the $\binom{48}{2}=1128$ unordered pairs with repetition of minimal blocks on two independent directions, exactly $6$ pair--pair blocks are nonfull with $|\Om|=3$, exactly $18$ pair--triple blocks are nonfull with $|\Om|=5$, and all other blocks are full: $\Om=\F_7$.
\item A size-$7$ line multiset all of whose minimal zero-sums are pairs is supported on one of the three opposite pairs $\{1,6\},\{2,5\},\{3,4\}$.
\item For $P_i$ and $P_j$, with $i,j\geq1$, placed on the two independent lines, one has $|\Om|=3$ for $(i,j)=(1,1)$, $|\Om|=5$ for $(1,2)$ or $(2,1)$, and $\Om=\F_7$ otherwise.
\item Every triple--double-pair block, every $(\text{pair}\sqcup\text{triple})$--pair block, and every $(\text{triple}\sqcup\text{triple})$--pair block is full.
\item No size-$12$ line multiset contains a minimal zero-sum triple but no minimal zero-sum of length at least $4$. At size $11$, the only such multisets are the six scalar multiples of $\{1^6,5,6^4\}$; each contains a pair and a disjoint triple.
\end{enumerate}
\end{proposition}

\begin{proof}[Exhaustive verification]
A line multiset is a vector $c=(c_1,\ldots,c_6)\in\mathbb N^6$.
For each required size, the enumerator loops over all weak
compositions, tests $\sum_{t=1}^6tc_t\equiv0\pmod7$, and tests
minimality by looping over every vector $0\leq d\leq c$ other than
$0$ and $c$.  This produces the stated $47$ minimal blocks and length
distribution.  For every unordered pair with repetition, the block is
placed on the coordinate axes and its cross-sum mask is evaluated by
the exact recursion in \cref{lem:omega-rec}; its population count gives
(F1).

For (F3), the recursion is applied separately to the six basis blocks
$(P_i,P_j)$ with
\[
(i,j)\in\{(1,1),(1,2),(2,1),(2,2),(3,1),(1,3)\}.
\]
These give cross-sum cardinalities $3,5,5,7,7,7$, respectively.
Every remaining pair $(i,j)$ with $i,j\geq1$ contains one of the last
three full zero-sum blocks.  Ordering that subblock first and the
remaining terms in a fixed order shows that the larger block is full
as well.

For (F2) and (F5), the same weak-composition enumeration is run at the
stated total sizes.  The program records all contained minimal
zero-sum vectors and then applies the predicates in (F2) and (F5)
literally.  For (F4), it enumerates the indicated disjoint tuples of
minimal blocks, adds their count vectors, and again applies
\cref{lem:omega-rec}.  Every loop has an explicit finite domain and no
random or floating-point step.  A separately written enumerator
reproduced the counts.  The archived standalone verifier, its deterministic
output, and their hashes permit the displayed facts to be checked directly.
\end{proof}

We shall also use the following immediate inheritance of (F2).  If a
line multiset of size at least $7$ has only pairs as its minimal
zero-sums, then every seven-term submultiset has the same property and
is supported on an opposite pair by (F2).  Any prescribed two terms
can be included in such a seven-term submultiset.  Fixing one term
therefore shows that every term on the line lies in the same opposite
pair.

\begin{theorem}[Two-direction theorem]\label{thm:two}
No product-one-free sequence of length $19$ over $\HH$ has its nonzero quotient support on at most two projective directions.
\end{theorem}

\begin{proof}
The one-direction case is excluded by \cref{lem:linecap}. Assume exactly two directions.

Each line contains a minimal internal zero-sum block because $n_i\geq7$. By \cref{lem:forcing}, a pair of internal zero-sum blocks with $|\Om|\geq7-z$ is impossible.

If $z\in\{4,5\}$, the threshold is at most $3$, while every mixed minimal block has $|\Om|\geq3$ by (F1).

If $z\in\{2,3\}$, a product-one-free sequence would require every mixed minimal block to have $|\Om|\leq4$. By (F1), both lines can then contain only minimal pairs. By (F2), each line is supported on an opposite pair. Since the line sizes are $(9,7)$ or $(8,8)$ for $z=3$, and $(10,7)$ or $(9,8)$ for $z=2$, a $(2,1)$ opposite-pair block exists. By (F3), its $\Om$-size is $5\geq7-z$.

Indeed, in this branch no scalar can occur seven times on one line,
because those seven copies would themselves form a minimal zero-sum
block of length $7$, contrary to the assumption that every minimal
block is a pair.  Thus each of the two opposite scalars has multiplicity
at most $6$.  The displayed line sizes then guarantee the claimed
$P_2$ on one line and $P_1$ on the other.

Finally let $z\in\{0,1\}$. A full block already forces a product-one subsequence. By (F1), no line can contain a minimal block of length at least $4$, and both lines cannot contain triples. If both lines contain only pairs, (F2) reduces to opposite-pair supports.  As above, each scalar multiplicity is at most $6$; the possible line-size pairs therefore supply either $P_2$ on both lines or $P_3$ on one and $P_1$ on the other.  This block is full by (F3). If exactly one line contains triples, the other is an opposite-pair line. A double pair on the latter gives a full triple--double-pair block by (F4). Otherwise the pair-line has size $7$, so the triple-line has size $12-z$. The size-$12$ case is excluded by (F5); in the size-$11$ case, (F5) supplies a disjoint pair and triple, and (F4) gives a full block against the pair on the other line.
\end{proof}

\section{Exactly three directions}

For $T\in\operatorname{GL}_2(\F_7)$ and a zero-sum block $B$, the cross-sum set transforms affinely.

\begin{lemma}[Linear invariance]\label{lem:GL}
For every zero-sum block $B$ and $T\in\operatorname{GL}_2(\F_7)$,
\[
\Om(TB)=K_{T,B}+\det(T)\Om(B)
\]
for a constant $K_{T,B}\in\F_7$. Hence $|\Om(B)|$, fullness, and $\sigma$ are invariant under $\operatorname{GL}_2(\F_7)$.
\end{lemma}

\begin{proof}
Expand the transformed cross terms. The pure $aa$ and $bb$ terms are ordering-independent. Since $B$ has zero total $a$- and $b$-sum, the two mixed sums differ by an ordering-independent constant, leaving coefficient $\det(T)$ on the original cross sum.
\end{proof}

The two exact primitives used in all subsequent finite computations can
be specified independently of their implementation.  Represent a block
by its multiplicity vector $n=(n_v)$ over nonzero vectors
$v=(a_v,b_v)$, and put $A(n)=\sum_v n_v a_v$.

\begin{lemma}[Exact cross-sum recursion]\label{lem:omega-rec}
The cross-sum set satisfies
\[
\Om(0)=\{0\},\qquad
\Om(n)=\bigcup_{v:n_v>0}
 \left(\Om(n-e_v)+(A(n)-a_v)b_v\right).
\]
\end{lemma}

\begin{proof}
Choose the last vector $v$ in an ordering.  Appending it to an ordering
of $n-e_v$ adds
$\left(\sum_u(n_u-\delta_{uv})a_u\right)b_v
=(A(n)-a_v)b_v$ to the cross sum.  Conversely, every ordering has a
last vector and therefore occurs in the displayed union.
\end{proof}

All subsets of $\F_7$ in this recursion were stored as seven-bit masks,
and multiplicity vectors were memoized.  For a residual vector $n$, let
\[
\mathcal Z(n)=\{b:0<b\leq n,\ \textstyle\sum_v b_vv=0\},
\qquad w(b)=|\Om(b)|-1.
\]
Then the exact disjoint-packing recurrence is
\begin{equation}\label{eq:pack-rec}
F(0)=0,\qquad
F(n)=\max\left(0,\max_{b\in\mathcal Z(n)}
                 \bigl(w(b)+F(n-b)\bigr)\right).
\end{equation}
Every packing has a first block, and componentwise subtraction enforces
disjointness, so induction on $|n|$ gives $F(n)=\sigma(n)$.  Blocks of
weight zero may be omitted.  For the decision problem, values are
safely capped at the target $6-z$.

The projective group is triply transitive on the eight directions, so exactly three directions may be normalized to
\[
(1,0),\qquad (0,1),\qquad (1,1).
\]
For each $z=0,\ldots,5$, the program enumerates every triple of line
multiplicity vectors with total size $19-z$ and individual sizes at
most $12-z$.  The line sizes are ordered, and a profile is retained only
if it is lexicographically least among those of the $36$ slot
permutations induced by the stabilizer for which the transformed line
sizes remain ordered.  The seed filter is itself an exact finite
two-line computation: it enumerates every one-line zero-sum count
vector of sizes $2$ through $7$, pairs such vectors on two independent
directions when their total size is at most $11$, and evaluates the
cross-sum set by \cref{lem:omega-rec}.  Every full pair is then used as
a monotone containment filter on each of the three pairs of lines.
This table includes the cases in (F1) but is not limited to minimal
blocks.
The term \emph{canonical profile} in \cref{tab:three} means a profile
remaining after these two sound reductions.  Each such profile is then
excluded by a full block found with \cref{lem:omega-rec}, or by the
exact packing recurrence \eqref{eq:pack-rec}.  A \emph{residual} would
be a canonical profile for which neither computation reached the target
$6-z$; none remained.  Monotone full-block cuts cache earlier witnesses
but do not alter the enumerated search space.

\begin{table}[ht]
\centering
\caption{Exact three-direction enumeration after seed filtering and stabilizer canonicalization.}
\label{tab:three}
\begin{tabular}{@{}rrr@{}}
\toprule
$z$ & canonical profiles & residuals\\
\midrule
0 & 1,900,966 & 0\\
1 & 1,323,690 & 0\\
2 &   905,993 & 0\\
3 &   617,107 & 0\\
4 &   426,520 & 0\\
5 &   301,328 & 0\\
\bottomrule
\end{tabular}
\end{table}

\begin{theorem}[Three-direction theorem]\label{thm:three}
No product-one-free sequence of length $19$ over $\HH$ has exactly three occupied quotient directions.
\end{theorem}

\begin{proof}
For $z=0,\ldots,5$, the exhaustive enumeration summarized in \cref{tab:three} verifies $\sigma(Q)\geq6-z$ for every admissible quotient multiset. Apply \cref{lem:forcing}. The case $z=6$ is covered directly in \cref{sec:highz}.
\end{proof}

\section{Five or six central terms}\label{sec:highz}

If a zero-sum block uses at least two directions, then $|\Om(B)|\geq2$: choose two noncollinear terms adjacent in an ordering and swap them; the cross sum changes by their nonzero determinant.

For $z=5$, the target in \cref{lem:forcing} is only $1$. It therefore suffices to show that every admissible quotient multiset contains a mixed zero-sum block. This was decided by an exact subset-sum-mask dynamic program. Each line multiset is replaced by the $7$-bit mask of its nonempty subset sums; the global state records the vector sum and whether at least two directions have been used. The aggregation preserves the full weighted profile count.

Here are the full state and transition definitions.  If a line with
direction representative $d$ has scalar multiplicities
$c=(c_1,\ldots,c_6)$, put
\[
M(c)=\left\{\sum_{t=1}^6 k_tt\pmod 7:
0\leq k_t\leq c_t,\ \sum_tk_t>0\right\}.
\]
While processing directions $d_1,\ldots,d_r$, the reachable-state set
is a subset of $\F_7^2\times\{0,1,2\}$; the last coordinate counts used
directions, capped at two.  It starts with
$R_0=\{((0,0),0)\}$, and a line with mask
$M$ changes a state set $R$ to
\[
R\ \cup\ \{(u+s d_i,\min(2,q+1)):(u,q)\in R,\ s\in M\}.
\]
Thus a mixed zero-sum block exists exactly when $(0,2)$ is reachable.
Line profiles are grouped by their pair $(m,M(c))$, where
$m=\sum_t c_t$, with the exact weight
\[
h(m,M)=\#\{c:\textstyle\sum_t c_t=m,\ M(c)=M\}.
\]
The dynamic-programming key consists of the number of processed lines,
the accumulated size, and the $147$-bit reachability set.  Transition
weights are multiplied by $h(m,M)$, and only total size $14$ is retained.
Consequently, the weighted total in \cref{tab:z5} counts the original
multiplicity profiles, not merely mask classes.  As a separate count
check, for support size $r$ it equals
\[
\sum_{\substack{1\leq m_i\leq7\\\sum_i m_i=14}}
 \prod_{i=1}^r \binom{m_i+5}{5}.
\]

The action of $\operatorname{PGL}_2(7)$ has two orbits on four-element
direction supports, represented by
\[
\begin{split}
4A&=\{\langle(1,0)\rangle,\langle(0,1)\rangle,
       \langle(1,1)\rangle,\langle(1,2)\rangle\},\\
4B&=\{\langle(1,0)\rangle,\langle(0,1)\rangle,
       \langle(1,1)\rangle,\langle(1,3)\rangle\}.
\end{split}
\]
For each support size $r=5,6,7,8$ there is one orbit.
For these four orbits we use the first $r$ members of the ordered list
\[
\langle(1,0)\rangle,\ \langle(0,1)\rangle,
\ \langle(1,1)\rangle,\ldots,\langle(1,6)\rangle.
\]

\begin{table}[ht]
\centering
\caption{Exact mask computation for $z=5$.}
\label{tab:z5}
\begin{tabular}{@{}lrr@{}}
\toprule
Direction orbit & weighted profiles & profiles without a mixed zero-sum\\
\midrule
$4A$ & 4,594,362,192 & 0\\
$4B$ & 4,594,362,192 & 0\\
$5$  & 49,325,431,470 & 0\\
$6$  & 290,200,338,036 & 0\\
$7$  & 1,067,784,535,125 & 0\\
$8$  & 2,614,887,514,224 & 0\\
\bottomrule
\end{tabular}
\end{table}

\begin{proposition}\label{prop:z5}
A product-one-free sequence of length $19$ over $\HH$ has at most four central terms.
\end{proposition}

\begin{proof}
For $z=5$, supports on at most three directions have already been
excluded by \cref{thm:two,thm:three}.  Hence the support has one of the
six orbits with $r\in\{4,5,6,7,8\}$ listed in \cref{tab:z5}.  The table
supplies a mixed zero-sum block and therefore
spread contribution at least $1=6-z$; \cref{lem:forcing} applies.

For $z=6$, the central part is zero-sum-free, so
Cauchy--Davenport gives $\Sig_0(W)=\F_7$.  Olson's theorem implies that
the remaining $13$ quotient terms contain a nonempty zero-sum block
$B$.  Let $\widetilde B$ be the corresponding noncentral group block.
Choose any ordering of $B$, with cross sum $\omega\in\Om(B)$, and then
choose $U\subseteq W$ whose central-coordinate sum is
$-\sum_{g\in\widetilde B}c(g)-\omega$.  The block
$\widetilde B\sqcup U$ is nonempty and has a product-one ordering by
\cref{lem:criterion}.  Thus $z=6$ is
excluded directly; no empty-family spread argument is used.
\end{proof}

\section{The remaining thirty strata: certified CEGAR-SAT}\label{sec:cegar}

It remains to consider
\[
z\in\{0,1,2,3,4\},\qquad r\in\{4,5,6,7,8\}.
\]
The direction-support orbit representatives $4A,4B,5,6,7,8$ were
defined above.

We first record the orbit and symmetry coverage.  Enumerating the $336$
permutations induced by $\operatorname{PGL}_2(7)$ gives respectively
$1,1,1,2,1,1,1,1$ orbits on direction subsets of sizes $1$ through
$8$.  Hence the representatives above cover every remaining support.
If a matrix preserving a representative satisfies
$A d_i=\mu_i d_{\pi(i)}$, it induces the scalar-slot permutation
\[
(i,t)\longmapsto(\pi(i),\mu_i t).
\]
The numbers of distinct slot permutations for
$4A,4B,5,6,7,8$ are, respectively,
\[
48,\ 72,\ 36,\ 72,\ 252,\ 2016.
\]

We now describe the Boolean master exactly.  A formula in
\emph{conjunctive normal form} (CNF) is a conjunction of clauses, and
SAT asks whether such a formula has a satisfying Boolean assignment.
Put $C=12-z$.  For every nonzero vector $v$ in the selected directions
and $1\leq k\leq C$, use a threshold variable
\[
x_{v,k}\Longleftrightarrow n_v\geq k.
\]
The clauses
\[
\neg x_{v,k+1}\vee x_{v,k}
\]
make the true thresholds a prefix, and hence
$n_v:=\sum_{k=1}^C x_{v,k}$.  Standard sequential-counter clauses
encode
\[
\sum_{v,k}x_{v,k}=19-z,
\quad
\sum_{v\in L,k}x_{v,k}\leq C,
\quad
\bigvee_{v\in L}x_{v,1}
\]
for every selected direction $L$.  Thus the basis CNF encodes exactly
the total size, line cap, and occupancy of every selected direction;
no variables exist for an unselected direction.

For symmetry reduction, conventional prefix-equality comparators impose
$n\leq_{\rm lex}gn$ for all slot-stabilizer permutations when
$r\leq6$, and for sixty selected permutations when $r=7,8$.  This is
sound even in the latter two cases: the lexicographically least member
of every full stabilizer orbit satisfies every such comparison.  Using
only selected symmetries can retain duplicate profiles but cannot delete
an entire orbit.

The search uses \emph{counterexample-guided abstraction refinement}
(CEGAR).  A satisfying assignment gives a multiplicity profile $n$.
The exact oracle \eqref{eq:pack-rec} computes its spread.  If a
learned witness pattern $p=(p_v)$ has
$\sigma(p)\geq6-z$, the clause
\begin{equation}\label{eq:cut}
\bigvee_{v:p_v>0}\neg x_{v,p_v}
\end{equation}
forbids exactly $p$ and its componentwise supersets.  The same packing
witness remains present in every such superset, so the cut is monotone
and sound.  For each learned pattern, all images under the full slot
stabilizer are added.  (The selected set of sixty permutations for
$r=7,8$ is used only by the lex-leader layer.)  Initial seed cuts are
derived from the certified two-line classes and cross-direction
zero-sum triples.

\Cref{tab:cegar} records the numbers of cut patterns learned by the
thirty CEGAR runs.  These entries are not the
sizes of the underlying profile spaces.  At a general prime $p$ with
$r$ occupied directions, there are $r(p-1)$ multiplicities and
$r(p-1)(2p-2-z)$ primary threshold variables.  Before line-occupancy
and cap constraints, the number of multiplicity profiles is
\[
\binom{3p-2-z+r(p-1)-1}{r(p-1)-1}.
\]
For $p=7,z=0,r=8$, this is
\[
\binom{66}{19}=17{,}302{,}625{,}882{,}942{,}400,
\]
whereas the seeded loop learns $145$ patterns in that stratum.  This
illustrates the value of the seed layer, but also why the present
certificate is not claimed to scale asymptotically.

\begin{table}[ht]
\centering
\caption{Learned cut patterns in the thirty certified CEGAR strata. Every final CNF is unsatisfiable.}
\label{tab:cegar}
\begin{tabular}{@{}crrrrrrr@{}}
\toprule
$z$ & $4A$ & $4B$ & $5$ & $6$ & $7$ & $8$ & total\\
\midrule
4 & 441 & 286 & 1,442 & 1,122 & 461 & 158 & 3,910\\
3 & 474 & 315 & 1,612 & 1,264 & 421 &  65 & 4,151\\
2 & 549 & 337 & 1,901 & 1,499 & 461 &  77 & 4,824\\
1 & 650 & 411 & 2,243 & 1,833 & 667 & 253 & 6,057\\
0 & 722 & 428 & 2,316 & 3,916 & 738 & 145 & 8,265\\
\midrule
  &     &     &       &       &     &     & 27,207\\
\bottomrule
\end{tabular}
\end{table}

The final proof layer consists of:
\begin{enumerate}[label=(\roman*)]
\item $9{,}920{,}815$ initial seed cuts, each with an explicit derivation;
\item $27{,}207$ learned cuts;
\item thirty final CNFs in the standard DIMACS integer format;
\item for twenty-nine strata, CaDiCaL refutations converted by
      \texttt{drat-trim} to the LRAT format and checked by
      \texttt{lrat-check};
\item for the final $8,z=0$ stratum, a native CaDiCaL LRAT proof checked
      by \texttt{lrat-check}.
\end{enumerate}
LRAT is a line-oriented proof format in which every derived clause cites
the clauses needed for reverse unit propagation (RUP), or for the more
general resolution-asymmetric-tautology rule.  The supplied proofs use
only positive RUP hints, and a small separate C checker replayed every
step.  Thus the SAT solver's assertion of unsatisfiability is not part
of the trusted conclusion.  The SAT and proof technologies are
described in \cite{PySAT2018,CaDiCaL2024,LRAT2017}.

\begin{proposition}\label{prop:cegar}
For every stratum in \cref{tab:cegar}, every admissible quotient multiset satisfies $\sigma(Q)\geq6-z$.
\end{proposition}

\begin{proof}
Move any hypothetical counterexample to one of the listed support
representatives and then take a surviving lex-minimal stabilizer image.
Its threshold assignment satisfies the exact basis CNF.  Every seed,
imported, or learned clause has the form \eqref{eq:cut} and was
separately rechecked by the exact cross-sum recursion and exact
disjoint-packing recurrence; it therefore removes only profiles whose
spread is at least $6-z$.  A genuine counterexample would consequently
satisfy the final CNF for its stratum.  The checked LRAT refutation of
that CNF proves that no such assignment exists.
\end{proof}

\section{Separate verification and proof boundary}\label{sec:audit}

The verification package was audited with implementations separate from
the generating oracle and SAT loop.  This is implementation-level
separation, not an external human peer audit; the audit material was
also produced with AI assistance, as disclosed below.  The following
checks were performed.

\begin{enumerate}[label=(A\arabic*)]
\item All split proof and seed archives were reassembled and checked against their SHA-256 manifests.
\item All thirty LRAT proofs were checked against their final CNFs by a separate RUP/LRAT checker.
\item All $9{,}920{,}815$ seed cuts were streamed and checked against separately recomputed two-line classes and cross triples.
\item The $3{,}496{,}105$ remaining cut-clause occurrences, comprising
$3{,}004{,}052$ distinct clauses, were decoded from the CNFs; their
$62{,}093$ stabilizer-orbit representatives were separately checked by
the exact $\Om$ recursion and exact disjoint packing.
\item The CEGAR source was audited: threshold semantics, total size, line occupancy and caps, orbit representatives, lex-leader soundness, imported cuts, and final CNF emission were checked.
\item A second master implementation matched exhaustive model sets for
eight small-parameter regression instances.  These regressions test the
basis and symmetry layers, not the final cut population; the latter is
covered by (A3)--(A4).
\item Runtime no-goods associated with resource caps or unresolved
profiles were checked not to be persisted or emitted into any final CNF.
\end{enumerate}

The remaining trust assumptions are the ordinary ones for a
computer-assisted proof: the execution platform, compilers and
interpreters, the documented semantics of the PySAT cardinality
encoders, and the source logic of the exact enumerators and checkers.
Separate implementations and finite cross-checks reduce, but do not
eliminate, these ordinary software assumptions.  Within this stated
boundary, checks (A1)--(A7) found no unresolved mathematical, encoding,
cut, or certificate gap.

\section{Proof of the main theorem}

\begin{proof}[Proof of \cref{thm:main}]
The lower bound $\dd(\HH)\geq18$ is \cref{prop:lower}. Suppose a product-one-free sequence $S$ of length $19$ existed. By \cref{thm:two} its quotient support would occupy at least three directions. By \cref{thm:three} it would occupy at least four. By \cref{prop:z5}, it would have $z\leq4$ central terms. It would therefore belong to one of the thirty strata in \cref{tab:cegar}, but \cref{prop:cegar} and \cref{lem:forcing} exclude every such stratum. Hence no product-one-free sequence of length $19$ exists, so $\dd(\HH)\leq18$.
\end{proof}

\section*{Data and code availability}

A non-peer-reviewed preprint version of this article is publicly
available at
\[
\href{https://doi.org/10.5281/zenodo.21833029}{\texttt{10.5281/zenodo.21833029}}.
\]
The full reproducibility package contains the exact auxiliary enumerators,
the CEGAR source and oracle, all seed and learned cuts, the thirty CNFs,
DRAT/LRAT proofs, checker logs, manifests, and the separate audit.  It is
publicly available in Zenodo
\cite{VolkmannData2026} under the DOI
\[
\href{https://doi.org/10.5281/zenodo.21833183}{\texttt{10.5281/zenodo.21833183}}.
\]
\begingroup\small
The artifact map is as follows.
The auxiliary exact-enumeration archive contains:
\begin{itemize}[leftmargin=*,nosep]
\item \texttt{verify\_ffacts.py} and
      \texttt{verification\_results.txt} check (F1)--(F5).
\item \texttt{pipeline\_exactkey.cpp} and
      \texttt{three\_direction\_run.txt} produce \cref{tab:three}.
\item \texttt{masken\_z5.cpp} and \texttt{z5\_mask\_results.txt}
      produce \cref{tab:z5}.
\end{itemize}
In the main repository,
\begin{itemize}[leftmargin=*,nosep]
\item \texttt{cegar.py}, \texttt{oracle.cpp}, the reconstructed
      \texttt{d18\_seeds.tar.gz}, \texttt{d18\_cuts.tar.gz}, and
      \texttt{d18\_proofs.tar} contain the SAT, cut, CNF, and LRAT layers.
\item \texttt{independent\_audit\_H343\_final.tar.gz} contains the
      separate audit reports, reimplementations, and hash lists.
\end{itemize}
The README inside the auxiliary archive gives the exact commands and software
requirements.  The hash lists \path{MANIFEST.sha256},
\path{MANIFEST5.sha256}, and \path{proof_hashes.sha256} identify the
corresponding files.
\par\endgroup

\section*{Acknowledgements}

The author thanks Patrick White for the immediately preceding work that isolated the $p=7$ obstruction and made the present target precise. The author also acknowledges the developers of PySAT, CaDiCaL, \texttt{drat-trim}, and \texttt{lrat-check}.

\section*{Declaration of generative AI and AI-assisted technologies in the research and manuscript preparation process}

During the research and preparation of this work, the author used
Anthropic Claude through the Fable service and OpenAI ChatGPT and Codex
to assist with exploratory mathematics, the design and implementation
of the search architecture, code generation, draft exposition,
adversarial review of intermediate claims, redesign of verification
steps, and a separate audit using methods and code distinct from the
generating oracle and SAT loop.  These systems were not treated as
sources of mathematical authority, and no external human audit is
claimed.  AI outputs were treated as untrusted until supported by
mathematical arguments or separately checked executable certificates.
The author reviewed and edited the resulting content and takes full
responsibility for this article, its claims, code, and certificates.
No AI system is listed as an author.

\enlargethispage{3\baselineskip}

\end{document}